\documentclass[11pt,a4paper]{article}
\usepackage[]{geometry}
\usepackage[mathcal]{euscript}
\usepackage{bm,url}                     
\usepackage{amsfonts}
\usepackage{amssymb}
\usepackage{amsmath}
\usepackage{amsthm}
\usepackage{graphicx}               
\usepackage{natbib}                 
\usepackage{colortbl}               
\usepackage{booktabs}
\newtheorem{definition}{Definition}
\newtheorem{remark}{Remark}
\newtheorem{theorem}{Theorem}

\newtheorem{lemma}{Lemma}

\newcommand{\bea}[1]{\textcolor{black}{#1}}

\usepackage[english]{babel}
\usepackage{amsmath}
\usepackage{amsfonts}
\usepackage{amssymb}
\usepackage{amsthm}
\usepackage{graphicx}
\usepackage{color}
\usepackage{array}
\usepackage{mathrsfs}
\usepackage{setspace}
\usepackage{todonotes}
\usepackage{hyperref}
\usepackage{float}

\newcommand{\eps}{\varepsilon}

\usepackage{amsfonts}
\usepackage{amssymb}
\usepackage{amsmath}
\usepackage{amsthm}

\usepackage{ulem}

\DeclareMathOperator{\reach}{reach}
\DeclareMathOperator{\unp}{Unp}

\begin{document}
	\emergencystretch 3em
	
	\begin{center}
		\Large \bf  On consistent estimation of dimension values 
	\end{center}
	\normalsize
	
	\
	
	\begin{center}
		Alejandro Cholaquidis$^a$, Antonio Cuevas$^b$, Beatriz Pateiro-Lopez$^{c}$, \\
		\footnotesize	$^a$ Centro de Matem\'atica, Universidad de la Rep\'ublica, Uruguay\\
		$^b$   Departamento de Matem\'aticas, Universidad Aut\'onoma de Madrid\\ and Instituto de Ciencias Matem\'aticas ICMAT (CSIC-UAM-UCM-UC3M)\\
		$^c$ Departamento de Estat\'{\i}stica, An\'alise Matem\'atica e Optimizaci\'on, Universidade \\ de Santiago de Compostela and Galician Center for Mathematical Research\\ and Technology  (CITMAga). 
		
	\end{center}

 \begin{abstract}
 	The problem of estimating, from a random sample of points, the dimension of a compact subset  $S$
 	of the Euclidean space is considered. The emphasis is put on consistency results in the statistical sense. That is, statements of convergence to the true dimension value when the sample size grows to infinity.
 	Among the many available definitions of dimension, we have focused (on the grounds of its statistical tractability) on three notions: the Minkowski dimension, the correlation dimension and  the, perhaps less popular, concept of pointwise dimension. 
 	We prove the statistical consistency of some natural estimators of these quantities. Our proofs partially rely on the use of an instrumental estimator formulated in terms of the empirical volume function $V_n(r)$, defined as the Lebesgue measure of the set of points whose distance to the sample is at most $r$.  In particular, we explore the case in which the true volume function $V(r)$ of the target set $S$  is a polynomial on some interval starting at zero. An empirical study is also included.  Our study aims to provide some theoretical support, and some practical insights, for the problem of deciding whether or not the set $S$ has a dimension smaller than that of the ambient space. This is a major statistical motivation of the dimension studies, in connection with the so-called ``Manifold Hypothesis''.  
 \end{abstract}

 \section{Introduction}

 \sloppy
 Let us assume that we have a random sample $X_1,\ldots,X_n$ drawn from a probability distribution $P$ whose support is an (unknown) compact set $S\subset {\mathbb R}^d$.  Our aim here is to study the statistical estimation of the dimension of $S$, where this term is understood in three different senses: Minkowski, correlation and pointwise dimension. They will be defined below, alongside with the classical notion of Hausdorff dimension which remains, in several aspects, as a sort of ``golden standard'' though, unfortunately, rather unsuitable for statistical treatment. This leads to consider other ``proxy notions'' of dimension more appropriate for statistical treatment. They all agree with Hausdorff dimension in regular cases. 
 
 A major statistical motivation for studying the estimation of the dimension of $S$ is to assess whether or not this dimension is that of the ambient space.
 
 So we place ourselves in the so-called ``Manifold Hypothesis'' setting, whose starting point is the empirical observation that many multivariate data sets found in practice are in fact confined into (or close to) a lower dimensional set. \\

 \noindent \textit{On the Manifold Hypothesis and the notion of dimension}
 
 In the context of high-dimensional statistics, the so-called Manifold Hypothesis (MH) is fulfilled when a cloud of points in the Euclidean space ${\mathbb R}^d$ lies in fact in (or is close to) a set (often a manifold) ${\mathcal M}$ whose dimension  is smaller than that of the ambient space.  The case where ${\mathcal M}$ is assumed to be linear leads to the classical theory of Principal Components Analysis which is a topic routinely covered in undergraduate courses of multivariate analysis.  But we are here concerned with the general, non-linear case. 
 
 A deep study of MH, within the differential geometry framework, is given in \citet{fef16}; see Section 2 in that paper for an overview on Manifold Learning. 
 
 Many other different strategies have been proposed to address, sometimes in an indirect fashion, the MH problem. These include (the list is largely non-exhaustive):
 
 \begin{itemize}
 	\item[] Fitting lower dimensional structures (curves or surfaces) to the data cloud. 
 	\item[]  Assessing lower-dimensionality (without explicit dimension estimation or surface fitting): \citet{aar17}, \citet{gen12a}, \citet{gen12b}, \citet{gen12c}.
 	
 	\item[]  Estimating the dimension of $S$. This is the approach we will follow here. More precisely, we aim at identifying, with probability one as the sample size tends to infinity, whether the support $S$ of the underlying probability measure of the data has a  dimension smaller than that of the ambient space. So far, this topic (and different variants of it) has been perhaps more popular in the engineering journals than in the statistical ones; see, however, \citet{bri13}, and references therein, for a statistically motivated approach. Different notions of dimension have been proposed in the literature: see, e.g, \citet{bis17} for an account oriented to fractal geometry. A survey paper with a more statistical approach can be found in \citet{cam16}. See also \citet{kim19}, \citet{qiu22} and \citet{blo22} for more recent contributions. 
 	
 	The notions of dimension considered here have been chosen  in account of their statistical tractability; see below for details. 
 \end{itemize}

 The contents of this work can be summarized as follows. 
 In Section \ref{sec:back}, some background is given on a few required geometric and statistical notions. 
 Section \ref{sec:notionsdim} provides a short account of a few notions of dimension currently used, with a particular focus on the aforementioned Minkowski, correlation  and pointwise dimensions. 
 In Section \ref{sec:estimators}, we define and motivate several estimators for the Minkowski, the correlation and the pointwise dimension of $S$. Some of them have been previously considered in the literature (see \citet{keg02}, \citet{you82}). Others, expressed in terms of volume functions, are mainly introduced as auxiliary tools in our asymptotic study.    All of them depend on a suitable sequence $r_n$ of smoothing parameters. 
 The main theoretical contributions of this paper are in Section \ref{sec:consistencia}, where mentioned consistency results are established. 
 An empirical study is commented in Section \ref{sec:emp}. 
 Some final conclusions are briefly highlighted in Section \ref{sec:conclusions}.

 \section{Some geometric and statistical background. Some notation}\label{sec:back}
 
 \noindent \textit{A few basic definitions}
 
 Given a set $S\subset \mathbb{R}^d$, we will denote by
 $\mathring{S}$ and $\partial S$ the interior and boundary of $S$,
 respectively, with respect to the usual topology of $\mathbb{R}^d$. The diameter of $S$ will be denoted as $\textnormal{diam}(S)$.

 The closed ball
 in $\mathbb{R}^d$,
 of radius $\varepsilon$, centred at $x$ will be denoted by $B(x,\varepsilon)$ and $\omega_d=\mu(B(x,1))$, $\mu$ being the Lebesgue measure on ${\mathbb R}^d$. 
 With a slight notational abuse, we denote $B(S,r)$ the  $r$-parallel set of $S$, ${B(S,r)=\{x\in{\mathbb R}^d: \inf_{y\in S}\|x-y\|\leq r\} }$, $\|\cdot\|$ being the Euclidean norm in ${\mathbb R}^d$. 
 
 \sloppy 
 Given two compact non-empty sets $A,C\subset{\mathbb R}^d$, 
 the \it Hausdorff distance\/ \rm or \textit{Hausdorff-Pompei} distance
 between $A$ and $C$ is defined by
 \begin{equation*}
 	\rho_H(A,C)=\inf\{\eps>0: \mbox{such that } A\subset B(C,\eps)\, \mbox{ and }
 	C\subset B(A,\eps)\}.\label{Hausdorff}
 \end{equation*}

 The following ``standardness'' notion appears, in slightly different versions, in the set estimation literature (see, e.g., \citet{cue04}):

 Given a probability measure $P$ with support $S\subset \mathbb{R}^d$, we say that $S$ is standard with respect to $P$ if 
 \ there exist positive constants $r_0$, $\delta$ and $d'$ such that, for all $x\in S$, 
 \begin{equation} \label{estandar}
 	P(B(x,r))\geq \delta r^{d'} \mbox{ and } r\in(0,r_0).
 \end{equation}

 A useful  tool in our approach will be the \textit{volume function} of $S$, which is defined, for $r\geq 0$, by $V(S,r)=V(r)=\mu(B(S,r))$. 
 The volume function plays a relevant role in geometric measure theory, as commented below. 
 

 \
 
 \noindent \textit{Federer's reach, polynomial volume and polynomial reach}
 
 Following \citet{fed59}, let us define the \textit{reach} of $S$ as the supremum  ${\mathbf r}$ of all values $R$ such that all points in $B(S,r)$ with $r\leq R$ have a unique metric projection onto $S$. In more formal terms,    let $\unp(S)$ be the set of points $x\in \mathbb{R}^d$ 
 with a unique metric projection onto $S$.
 For $x\in S$, let $\reach(S,x)=\sup\{r>0:\mathring{B}(x,r)\subset \unp(S)\big\}$. 
 The \textit{reach} of $S$ is then defined by $\reach(S)=\inf\{\reach(S,x): x\in  S\}$, and $S$
 is said to be of positive reach if ${\mathbf r}=\reach(S)>0$. In this case, it is shown in \citet{fed59} that the volume function $V(r)$ is a polynomial on the interval $[0,{\mathbf r}]$,
 \begin{equation*}
 	V(r)=\theta_0+\theta_1r+\ldots+\theta_dr^d,\ \mbox{for all } r\in[0,\mathbf r].\label{pvol}
 \end{equation*}	
 Also, 
 the coefficients of this polynomial have a relevant geometric information on $S$: in particular, $\theta_{0}=\mu(S)$, $\theta_1$ is (outer) Minkowski measure of the boundary of  $S$, $\theta_d$ is, up to a known factor, the Euler-Poincar\'{e} characteristic of $S$ and the remaining coefficients can be interpreted in terms of curvatures.

 Still, it is important to note that a polynomial volume expression for $V(r)$ can hold even if $\reach(S)=0$. For instance it holds for the subset $S=[-1,1]^2\setminus [-1/2,1/2]^2$  of $\mathbb{R}^2$. This polynomial volume expression motivates the following definition given in \citet{ch23}, see also \citet{ch14}.

 \begin{definition}\label{def:polvol}
 	Given a compact set $S\subset {\mathbb R}^d$ with volume function $V(r)=\mu(B(S,r))$, we  define the \textit{polynomial reach} ${\mathbf R}$ of $S$ as 
 	\begin{equation*}\label{pol_reach}
 		{\mathbf R}=\sup\{R\geq 0:V(r)\mbox{ is a polynomial of degree at most $d$  on $[0,R]$}\}.
 	\end{equation*}
 \end{definition}

 %
 
 \

 \section{Different notions of  dimension}\label{sec:notionsdim}

 Many proposals have been put forward to formally define the intuitively based notion of dimension of a set $S\subset {\mathbb R}^d$. A very short, partial account is included below. We start by mentioning the Hausdorff dimension (which, in several aspects is considered as a standard reference). Then we focus on the Minkowski dimension (on account of its statistical tractability) and we consider as well the notions of correlation dimension and pointwise dimension. The statistical estimation of these quantities will be addressed in subsequent sections. 
 
 \
 
 \noindent \textit{Hausdorff dimension}
 
 The Hausdorff  dimension, first introduced by  \citet{haus19}, is perhaps the most widely recognized member of the family of fractal dimensions whose aim is to quantify the complexity of the geometry of a set, its scaling  and self-similarity properties; see \citet{fal04} for details.  Its formal definition is based on the notion of $\alpha$-dimensional Hausdorff content, given by
 \[\mathcal{H}_\infty^\alpha(S)=\inf\left\{\sum_{i} \textnormal{diam}(U_i)^\alpha:\ S \subset\bigcup_i U_i \right\},\]
 where $\{U_i\}$ is a countable collection of sets that cover $S$.
 Then, the Hausdorff dimension is defined as
 \[\dim_\textnormal{H}(S)=\inf\{\alpha:\ \mathcal{H}_\infty^\alpha(S)=0\}.\]
 
 Hausdorff dimension is, in several aspects,  a sort of  ``ideal reference'' to define dimension as it enjoys a number of desirable properties, not necessarily shared by other dimension notions. However, in practice, $\dim_\textnormal{H}(S)$ could be very difficult to compute. Thus, in practical applications, there is a case to consider other dimension notions, as those considered below, as proxies for $\dim_\textnormal{H}$.
 
 \
 
 \noindent \textit{The concept of Minkowski dimension. Some equivalent definitions}
 
 Let $S\subset \mathbb{R}^d$ be a bounded set. Define the \textit{covering number}, $N(S, r)$ to be the minimal number of sets of diameter at most  $r$ required to cover $S$.
 
 Then, the Minkowski dimension, also known as Minkowski-Bouligand dimension, Kolmogorov capacity, Box-counting dimension, or entropy dimension, is defined as 
 \begin{equation}\label{bcdim}
 	\dim(S)=\lim_{r\to 0 }\frac{\log(N(S, r))}{\log(1/r)}.
 \end{equation}
 To motivate this definition in intuitive terms, note that \eqref{bcdim} means that the dimension is the exponent $k$ such that $N(S, 1/n)\approx Cn^k$.  
 
 We will assume throughout that the limit in \eqref{bcdim} exists. If this were not the case, most results can be re-written in terms of the upper and lower Minkowski dimension, defined in terms, respectively, of the upper or lower limit in \eqref{bcdim}; see, e.g.,  \citet{fal04}.

 Some alternative, equivalent  expressions for the Minkowski dimension  can be obtained by replacing in \eqref{bcdim} the covering number $N(S,r)$ with either the \textit{packing number} $N_{\textnormal{pac}}(S,r)$ or the \textit{separating number}  $N_{\textnormal{sep}}(S,r)$ defined, respectively, as the maximum possible cardinality of a disjoint collection of closed $r$-balls with centres on $S$ and the maximum cardinality of an \textit{$r$-separated} subset of $S$ (where $X\subset S$ is said to be $r$-separated if  \(x, y \in X\) implies \(\| x - y \| \geq r\)). 
 
 The equivalence of these alternative definitions follows from the relations
 \begin{equation}\label{enes}
 	N(S,4r)\leq N_{\textnormal{sep}} (S,2r)\leq N(S,r)\  \mbox{and } N_{\textnormal{pac}}(S,r)=N_{\textnormal{sep}}(S,2r),
 \end{equation}
 which hold for any bounded set $S$ in the Euclidean space,  see \citet[p. 67]{bis17}. 
 
 A further alternative expression for $\dim(S)$ is
 \begin{equation}\label{bcdim2}
 	\dim(S)=d-\lim_{r\to 0 }\frac{\log(\mu(B(S,r)))}{\log(r)}=d-\lim_{r\to 0 }\frac{\log(V(r))}{\log(r)},
 \end{equation}
 provided that this limit exists (it can be $+\infty$). The equivalence between \eqref{bcdim} and \eqref{bcdim2} follows from \eqref{enes}  together  with the simple inequalities 
 \begin{equation}\label{mun}
 	\mu(B(S,r))\leq N(S,r)\mu(B(0,r)), \ \mu(B(S,r))\geq N_{\textnormal{sep}}(S,r)\mu(B(0,r/2)).
 \end{equation}
 
 The relationship between Minkowski dimension and Hausdorff dimension is given by
 $\dim_\textnormal{H}(S)\leq \dim(S)$,
 where strict inequality is possible. A simple example is  $\mathbb{Q}$, the set of rational numbers where $\dim_\textnormal{H}(\mathbb{Q})=0<1= \dim(\mathbb{Q})$. 
 
 \
 
 \noindent \textit{The correlation dimension}

 Another popular method for measuring some sort of fractal dimension is the \textit{correlation dimension}, introduced by \citet{gras83}; see also \citet{cam16} for a survey and \citet{pes93} for some mathematical insights. In fact, the definition below follows the formal treatment in \citet{pes93}, rather than the original formulation in \citet{gras83}. 
 
 
 Let $X_1,X_2$ be two independent and identically distributed (iid) copies of a random element $X$ in $\mathbb{R}^d$. Let us define
 \begin{equation*}\label{cd}
 	p(r)=\mathbb{P}(\|X_1 - X_2\|<r).
 \end{equation*}
 Then, the correlation dimension of the distribution of $X$, $P$, is defined as
 \begin{equation}\label{cd2}
 	\dim_\textnormal{cd}(P)=\lim_{r\to 0} \frac{\log(p(r))}{\log(r)},
 \end{equation}
 provided that this limit exists (it can be $+\infty$). Although we restrict ourselves to \( \mathbb{R}^d \), the correlation dimension is formally defined in \citet{pes93} for a general metric space.
 
 An outstanding difference between the definition of  correlation dimension in \eqref{cd2} and other notions of dimension is the presence of the probability measure $P$.  Of course, in our case, to make \eqref{cd2} comparable to other dimension notions, we will focus on probability measures $P$ whose support $S$ is the set of interest. In the Supplementary Material it is shown that in fact the limit in \eqref{cd2} is the same for all $P$ fulfilling some regularity conditions. 
 
 It is perhaps worth noting that, when the support $S$ is a compact manifold in ${\mathbb R}^d$, under some regularity conditions, the norm $\Vert\cdot\Vert$ of the ambient space can be replaced with the geodesic distance $d_S$ in $S$ still obtaining the same result in \eqref{cd2}. Indeed, since, in general 
 $d_S(x,y)\geq \Vert x-y\Vert$, it will suffice to have 
 $d_S(x,y)\leq C\Vert x-y\Vert$, for some constant $C>0$.
 This is guaranteed, under the above mentioned condition of positive reach \citep{fed59}: see, e.g., \cite[Lemma 3]{gen12b}.  

 \
 
 \noindent \textit{The pointwise dimension}
 
 The so-called pointwise dimension $\dim_\textnormal{pw}(x)$ (see \citet{you82}, \citet{cam16}) differs from Minkowski dimension $\dim(S)$ in at least two aspects: first, again, it depends on the underlying probability distribution of the data points, rather than simply on the support $S$. Second, it is defined point-by-point so that  it takes into account as well the local aspects. Thus $\dim_\textnormal{pw}(x)$ provides information about different regions of the data cloud for which the global Minkowski dimension is blind. 
 
 If $P$ is a probability measure with support $S\subset {\mathbb R}^d$, we define the pointwise dimension of $P$ at $x\in S$ as
 \begin{equation}\label{eq:pw}
 	\dim_{\textnormal{pw}}(x)=\lim_{r\to 0} \frac{\log(P(B(x,r)))}{\log(r)},
 \end{equation} 
 provided that this limit does exist (it can be $+\infty$). While, obviously, $\dim_{\textnormal{pw}}(x)$ depends on the probability $P$, it is clear that many different choices of $P$ will lead to the same results. For example, $\dim_{\textnormal{pw}}(x)$ will equal $q$ for all choices of $P$ such that $P(B(x,r))$ is of exact order $r^q$.

 If $S$ is a compact Riemannian manifold, a natural choice for $P$, aimed at giving a sort of  ``intrinsic'' pointwise dimension notion for $S$, would be the uniform distribution with respect to the volume form; see \citet[Section 1.3]{pen06}.
 
 A natural way of deriving a ``global'' notion of dimension for a set $S$ from \eqref{eq:pw} would be just defining $\dim_\textnormal{pw}(S)=\sup_{x\in S}\dim_\textnormal{pw}(x)$; see Section \ref{sec:estimators} for more discussion on this.  It can be seen \citep[Prop. 2.1]{you82} that when $P$ is a probability measure with support on $S$, $\dim_\textnormal{H}(S)\leq \dim_\textnormal{pw}(S)$. Also \citep[Th. 4.4]{you82} if $P$ is a probability measure on a compact Riemannian manifold and $\dim_\textnormal{pw}(x)=q$ almost surely, then $\dim_\textnormal{H}(S)=q$.

 \

 
 %
 
 \section{Some estimators of the considered dimension notions}\label{sec:estimators}

 Our basic aim here is to define consistent estimators  for the different notions of dimension introduced in the previous section.  All these estimators are based on a random sample $\aleph_n=\{X_1,\dots,X_n\}$ of points from a probability distribution whose support is $S\subset\mathbb{R}^d$. The consistency of these estimators will be established in Section \ref{sec:consistencia}. Here, the term ``consistency'' must be understood in the statistical sense: we want that our estimators converge (almost surely) as $n$ grows to infinity.  According to the usual paradigm in nonparametrics, all the proposed estimators will depend on a real sequence $r_n$ of smoothing parameters which must tend to zero slowly enough in order to achieve consistency. 
 
 \
 
 \noindent \textbf{The capacity estimator.}  A first natural estimator for the Minkowski dimension would result from definition \eqref{bcdim}, 
 \begin{equation}\label{dimcap_est}
 	\widehat{\dim}_{\textnormal{cap}}=-\frac{\log(N_{\textnormal{sep}}(\aleph_n,r_n))}{\log(r_n)},
 \end{equation}
 where  $N_{\textnormal{sep}}(\aleph_n,r_n)$ is the natural empirical estimator of the separating number $N_{\textnormal{sep}}(S,r)$, that is, the maximum cardinality of an $r_n$-separated set in the sample $\aleph_n$, and $r_n$ is an appropriate sequence of smoothing parameters $r_n\downarrow 0$. 
 
 This estimator was previously considered in \citet{keg02}; se also \citet{cam16}. We keep the name ``capacity estimator'' and the subindex ``cap'' in \eqref{bcdim} to keep K\'egl's notation although, according to the general notation we have followed so far (borrowed from \citet{bis17}) the sub-index ``sep'' would be  acceptable as well. In fact, the main contribution in \citet{keg02} is an efficient algorithm to calculate the estimator \eqref{dimcap_est} or, more precisely, a ``scale-dependent'' version of it. For simplicity and ease of presentation, this version will not be considered in our consistency results below which, in any case, can be easily adapted to it. 
 
 \

 \noindent \textbf{A non-parametric volume-based estimator.} 
 An alternative approach is obtained by simply plugging-in the volume function in \eqref{bcdim2} by its empirical counterpart ${V_n(r_n)=\mu(B(\aleph_n,r_n))}$,
 for some appropriate sequence of smoothing parameters $r_n\downarrow 0$. This leads to 
 \begin{equation}\label{dimvol_est}
 	\widehat{\dim}_{\textnormal{vol}}=d-\frac{\log(V_n(r_n))}{\log(r_n)}.
 \end{equation}
 
 Lemma \ref{lem1} below establishes that this estimator differs from K\'egl's estimator by less than $1/\log(r_n)$, up to known constants.

 \

 \noindent \textbf{An estimator of the correlation dimension.}  
 Let us define 
 $$\widehat{p}_n(r)=\binom{n}{2}^{-1}\sum_{i\neq j}\mathbb{I}_{\{\|X_i-X_j\|<r\}},$$ 
 where $\mathbb{I}$ denotes an indicator function.  Observe that $\mathbb{E}(\widehat{p}_n(r))=p(r)$. Then, we can consider as  estimator for the correlation dimension

 \begin{equation}\label{dimcd_est}
 	\widehat{\dim}_{\textnormal{cd}}=\frac{\log(\hat p_n(r_n))}{\log(r_n)}.
 \end{equation}

 \
 
 \noindent \textbf{A plug-in estimator of the pointwise dimension.}  An empirical version of the pointwise dimension $\dim_\textnormal{pw}$, defined in \eqref{eq:pw}, is given in a natural way by replacing the probability measure $P$ by its empirical counterpart, ${\mathbb P}_n$:
 \begin{equation}\label{estpw}
 	\widehat{\dim}_{\textnormal{pw}}(x)=  \frac{\log({\mathbb P}_n(B(x,r_n)))}{\log(r_n)}.
 \end{equation}
 \ As it follows from the discussion in Section \ref{sec:notionsdim}, a global estimator of the dimension of $S$ could be obtained from \eqref{estpw} simply defining $\widehat{\dim}_{\textnormal{pw}}(S)=\sup_{x\in S}\widehat{\dim}_{\textnormal{pw}}(x)$. However, in practice, we cannot calculate the estimator in all points of $S$. So, in the simulation outputs of Section \ref{sec:emp}, we have used  the 0.9 quantile of the values $\widehat{\dim}_{\textnormal{pw}}(X_i)$. Of course, the motivation for this is to have some protection against outliers.

 \
 
 \section{Consistency results}\label{sec:consistencia}

 \subsection{Consistency for the volume-based estimator and K\'egl's estimator}\label{sec:consistkegl}
 
 The following technical lemma establishes a relationship between K\'egl's estimator and the volume-based estimator. It will be used in our two main theorems to derive results about K\'egl's estimator using the volume-based estimator.

 \begin{lemma} \label{lem1} Let $r_n$ be a sequence such that $0<r_n<1$ and $r_n\to 0$, for any $n\in\mathbb{N}$
 	\begin{equation}\label{lem:ineq}
 		\Big|\widehat{\dim}_{\textnormal{vol}}-\widehat{\dim}_{\textnormal{cap}}+\frac{\log(\omega_d)}{\log(r_n)}  \Big|\leq -d\frac{\log(2)}{\log(r_n)}\qquad a.s.
 	\end{equation}	
 \end{lemma}
 \begin{proof} From \eqref{mun} and $N(K,4r)\leq N_{\textnormal{sep}} (K,2r)\leq N(K,r)$ (these inequalities are valid for any bounded set $K$ in the Euclidean space; see \citet[p. 67]{bis17})\bea{, we have}
 	$$N_{\textnormal{sep}}(\aleph_n,r_n)\omega_d (r_n/2)^d\leq V_n(r_n)\leq V_n(2r_n)\leq N_{\textnormal{sep}}(\aleph_n,r_n)\omega_d(2r_n)^d.$$
 	Then,
 	\begin{align*}
 		\log(\omega_d)+ d\log(r_n) -d\log(2)&\leq \log(V_n(r_n))-\log(N_{\textnormal{sep}}(\aleph_n,r_n)) \\
 		&\leq  \log(\omega_d) +d\log(r_n) +d\log(2).
 	\end{align*}
 	Dividing all terms by $-\log(r_n)$ results in
 	$$-\frac{\log(\omega_d)}{\log(r_n)}  +d\frac{\log(2)}{\log(r_n)}\leq \widehat{\dim}_{\textnormal{vol}}-\widehat{\dim}_{\textnormal{cap}}\leq  -\frac{\log(\omega_d)}{\log(r_n)}  -d\frac{\log(2)}{\log(r_n)},$$
 	from where it follows \eqref{lem:ineq}.
 \end{proof}

 The following result provides conditions for the (almost sure) consistency, as $n\to\infty$ of the estimators  \eqref{dimcap_est} and \eqref{dimvol_est}. 
 
 \begin{theorem}\label{th:general}  Let $S\subset \mathbb{R}^d$ be a compact set such that   $V(r)$ is Lipschitz in some interval $[0,\lambda]$ with $\lambda>0$. Let $\aleph_n=\{X_1,\dots,X_n\}$ be an iid sample from a distribution whose support is $S$. Let $\gamma_n=\rho_H(\aleph_n,S)$ and $r_n\to 0$ such that    $\gamma_n<r_n$ and ${\gamma_n/(V(r_n-\gamma_n)\log(r_n))\to 0}$, almost surely.	Then, 
 	
 	(a) the estimator $\widehat{\dim}_{\textnormal{vol}}$ defined in \eqref{dimvol_est} is almost surely consistent, that is
 	\begin{equation*}\label{consist-dvol}
 		\dim(S)=d-\lim_{n\to \infty }\frac{\log(V_n(r_n))}{\log(r_n)}\quad a.s.,
 	\end{equation*}
 	where $V_n(r_n)=\mu(B(\aleph_n,r_n))$ is the natural empirical estimator of $V(r_n)$.

 	(b) K\'egl's estimator $\widehat{\dim}_{\textnormal{cap}}$, defined in \eqref{dimcap_est} is almost surely consistent as well, under the same conditions for $r_n$. 
 	
 \end{theorem}
 
 \begin{proof}  
 	(a) Let us write, 	
 	$$ \frac{\log(V_n(r_n))}{\log(r_n)}= \frac{\log(V(r_n))-\log(V(r_n)/V_n(r_n))}{\log(r_n)}.$$
 	From \eqref{bcdim2}, we only have to prove that 
 	\begin{equation}\label{eqth1}
 		\frac{\log(V(r_n)/V_n(r_n))}{\log(r_n)}\to 0.
 	\end{equation}
 	
 	If $r_n>\gamma_n$, from the first equation in the proof \citet[Prop. 4.2]{ch23}
 	$V(r_n-\gamma_n)\leq V_n(r_n)$. Then,  using that $\log(x)\leq x-1$  for $x\geq 1$,
 	\begin{equation}\label{rgamma}
 		\Bigg|\frac{\log(V(r_n)/V_n(r_n))}{\log(r_n)}\Bigg|\leq \Big|\frac{\log(V(r_n)/V(r_n-\gamma_n))}{\log(r_n)}\Bigg|\leq \Big|\frac{V(r_n)-V(r_n-\gamma_n)}{V(r_n-\gamma_n)\log(r_n)}\Big|.
 	\end{equation}
 	Now, from the Lipschitz assumption on $V$ in $[0,\lambda]$, there exists $L>0$ such that ${|V(r_n)-V(r_n-\gamma_n)|\leq L\gamma_n}$ for all $n$ large enough such that $r_n<\lambda$.  So the right-hand side of \eqref{rgamma} is of order $\gamma_n/(V(r_n-\gamma_n)\log(r_n))$
 	and \eqref{eqth1} follows from ${\gamma_n/(V(r_n-\gamma_n)\log(r_n))\to 0}$. 
 	
 	\
 	
 	\noindent (b) The consistency of $\widehat{\dim}_{\textnormal{cap}}$ follows from (a) and Lemma \ref{lem1}. 
 	
 \end{proof}
 
 \begin{remark}\label{rem:st}
 	In order to see the true extent of the assumptions in Theorem \ref{th:general} let us note that, under the standardness assumption
 	\eqref{estandar} it is proved in \citet[Th. 3]{cue04} that 
 	$$
 	\gamma_n=\rho_H(\aleph_n,S)=O\left(\left(\frac{\log n}{n}\right)^{1/d}\right),
 	$$
 	with probability one.
 	Therefore, the condition $\gamma_n<r_n$, a.s. would hold whenever $r_n=O\left(\left({\log n}/{n}\right)^{1/d'}\right)$, for some $d'>d$. 
 	
 	The other assumption, $\gamma_n/(V(r_n-\gamma_n)\log(r_n))\to 0$ a.s., is a bit more delicate, as it involves the behaviour of $V$ near zero. 
 \end{remark}

 \begin{remark}\label{rem:sobreVn}
 	
 	The volume-based estimator in Theorem \ref{th:general} above plays here a somewhat instrumental role in order to get in part (b) the consistency result for $\widehat{\dim}_{\textnormal{cap}}$. 
 	Indeed, in principle, the computation of $\widehat{\dim}_{\textnormal{vol}}$ is more expensive, especially taking into account the efficient algorithm provided in \citet{keg02} to calculate $\widehat{\dim}_{\textnormal{cap}}$. 
 \end{remark}
 
 \subsection{Some results under the polynomial volume assumption}\label{sec:polvol}

 If we assume that $V(r)=\mu(B(S,r))$ is a polynomial on some interval $[0,{\mathbf R}]$, the Minkowski dimension is always  $d$ minus an integer corresponding to the order of the first non-null coefficient in the polynomial $V$. Also, under an additional shape restriction on $S$, we have  a quite precise guide about the choice of the sequence of smoothing parameters $r_n$ in the plug-in consistent estimator $V_n(r_n)$ considered in Theorem \ref{th:general}. Finally, the polynomial volume assumption provides an alternative natural estimator of $V(r)$, denoted ${\mathcal P}_n(r)$,  constructed by minimizing the $L^2-$distance between  $V(r)$ ant the empirical volume function $V_n(r)$ introduced in the previous subsection. These ideas are formalised in the following result.

 \begin{theorem}\label{th:polvol}
 	Let $S\subset \mathbb{R}^d$ be a compact set with polynomial volume function $V(r)=\sum_{j=k}^d \theta_j r^{j}$, $r\in [0,\mathbf{R}]$, $\theta_k$ being  the first non null coefficient in $V(r)$.  Given a sample $\aleph_n=\{X_1,\ldots,X_n\}$ drawn on $S$, denote ${\mathcal P}_n(r)=\sum_{j=k}^d \hat\theta_j r^{j}$, where $\hat\theta_j$ stand for the minimum-distance estimators of $\theta_j$ based on the $L^2$-distance   between $V_n(r)$ and $V(r)$   on the interval $[0,{\mathbf R}]$. Then,

 	(a) $\dim(S)=d-k$.

 	(b)   If $\gamma_n=\rho_H(\aleph_n,S)$ and $r_n\to 0$ is such that $\gamma_n/r_n\to 0$, a.s., then  we have that, for  $n$ large enough,
 	\begin{equation}\label{eq:conVn}
 		\Big|\dim(S)-\widehat{\dim}_{\textnormal{vol}}\Big |\leq \frac{|\log(2\theta_k)|}{|\log(r_n)|}\qquad  \mbox{a.s}.
 	\end{equation}
 	Moreover, under the standardness condition $P(B(x,r))\geq \delta r^d$ introduced in \eqref{estandar}, where $P$ stands here for the distribution of the $X_i$, condition $\gamma_n/r_n\to 0$ is fulfilled for any sequence $r_n$ of type
 	$$
 	r_n=\Big(\frac{\log n}{n}\Big)^{1/d^\prime{}}
 	$$
 	with  $d^\prime>d$. As a consequence, for this choice of $r_n$ we also have  $|\dim(S)-\widehat{\dim}_{\textnormal{vol}}|=O(1/\log(n))$. 
 	
 	(c)  Assuming again $r_n\to 0$ and $\gamma_n/r_n\to 0$, a.s.,  K\'egl's estimator \eqref{dimcap_est} fulfils 
 	\begin{equation*}
 		\Big|\dim(S)-\widehat{\dim}_{\textnormal{cap}}\Big |\leq \frac{|\log(2\theta_k)|+|\log(\omega_d)|+d\log(2)}{|\log(r_n)|}  \qquad  \mbox{a.s}.
 	\end{equation*}

 	(d) Let $f(x)=\lfloor x+1/2\rfloor$ be the function which maps any positive value $x>0$ to its nearest integer value. Then, there exists $r_0>0$ such that the estimator 
 	\begin{equation*}\label{esttilde}
 		\widetilde{\dim}=f\Big(d-\frac{\log\big({\mathcal P}_n(r_0))}{\log (r_0)}\Big),
 	\end{equation*}
 	fulfils $\widetilde{\dim}=\dim(S)$ eventually almost surely, as $n\to\infty$.

 \end{theorem}
 \begin{proof}
 	(a) Observe that a simple application of L'H\^opital rule yields
 	\begin{equation*}\label{bc1}
 		\dim(S)=d-\lim_{r\to 0} \frac{\log(V(r))}{\log(r)}=d-\lim_{r\to 0}\frac{\log(\sum_{j=k}^d \theta_j r^{j})}{\log(r)}=d-k.
 	\end{equation*}

 	(b) 	 By part (a) and the polynomial volume assumption, for $n$ large enough such that $r_n< \mathbf R$,
 	\begin{align*}
 		\Big|\dim(S)-\widehat{\dim}_{\textnormal{vol}}\Big|= & \Bigg| \frac{\log(V_n(r_n)/V(r_n))}{\log(r_n)} + \frac{\log\Big(\sum_{j=0}^{d-k} \theta_{j+k} r_n^{j}\Big)}{\log(r_n)}\Bigg|\\
 		\leq &  \Bigg| \frac{\log(V_n(r_n)/V(r_n))}{\log(r_n)} \Bigg|+\Bigg|\frac{\log\Big(\sum_{j=0}^{d-k} \theta_{j+k} r_n^{j}\Big)}{\log(r_n)}\Bigg|.
 	\end{align*}
 	Let us bound the first term,  using the same bounds used to prove Theorem  \ref{th:general} (a)
 	$$\Bigg|\frac{\log(V_n(r_n)/V(r_n))}{\log(r_n)}  \Bigg|	\leq \Bigg|\frac{\log(V(r_n)/V(r_n-\gamma_n))}{\log(r_n)}\Bigg|.$$
 	From the polynomial volume assumption and writing 
 	$$(r_n-\gamma_n)^j=r_n^j+\sum_{l=0}^{j-1}\binom{j}{l} r_n^l\gamma_n^{j-l}(-1)^{j-l}\quad j=k,\dots,d,$$
 	we get that $V(r_n-\gamma_n)=V(r_n)+\mathcal{Q}_n(\gamma_n)$ where $\mathcal{Q}_n(\gamma_n)$ is a polynomial in $\gamma_n$, that depends on $n$, but whose  degree is at most $d$. Observe  also that the independent term of $\mathcal{Q}_n$ is 0.  Now if we use that $\log(x)\leq x-1$ for $x\geq 1$,
 	\begin{align*}
 		\Bigg|\frac{\log(V(r_n)/V(r_n-\gamma_n))}{\log(r_n)}\Bigg|&\leq \frac{|\mathcal{Q}_n(\gamma_n)|}{|(V(r_n)+\mathcal{Q}_n(\gamma_n))\log(r_n)|}\\
 		&=\frac{1}{|(1 +V(r_n)/\mathcal{Q}_n(\gamma_n))\log(r_n)|}.\\
 	\end{align*}
 	Now, since the independent term of $\mathcal{Q}_n$ is 0 and $r_n/\gamma_n\to \infty$, it follows that ${|(1 +V(r_n)/\mathcal{Q}_n(\gamma_n))|\to\infty}$ as $n\to\infty$. Then, for $n$ large enough, 
 	$$\frac{1}{|(1 +V(r_n)/\mathcal{Q}_n(\gamma_n))\log(r_n)|}+\Bigg|\frac{\log\Big(\sum_{j=0}^{d-k} \theta_{j+k} r_n^{j}\Big)}{\log(r_n)}\Bigg| \leq \frac{|\log(2\theta_k)|}{|\log(r_n)|}.$$
 	
 	The statement concerning the standardness assumption follows from \citet[Th. 3]{cue04}. This result establishes that under condition \eqref{estandar}, ${\gamma_n=O\Big((\log n/n)^{d}\Big)}$. It can be noted that for this choice of $r_n$, bound \eqref{eq:conVn} yields
 	$ |\dim(S)-\widehat{\dim}_{\textnormal{vol}}|=O(1/\log(n))$.

 	(c) The proof follows directly from part (b) and Lemma \ref{lem1}.

 	(d) Assume $\dim(S)=d-k$.  For all $0<r\leq \mathbf{R}$,
 	$$V(r)=  r^k\Bigg(\sum_{j=k}^d \theta_j r^{j-k}\Bigg).$$
 	Now,
 	\begin{align*}
 		f\Bigg(d-\frac{\log({\mathcal P}_n(r))}{\log(r)}\Bigg)=&f\Bigg(d-\frac{\log(V(r))-\log(V(r)/{\mathcal P}_n(r))}{\log(r)}\Bigg)\\
 		=&f\Bigg(d-k-\frac{\log(\sum_{j=k}^d \theta_j  r^{j-k})}{\log(r)} +\frac{\log(V(r)/{\mathcal P}_n(r))}{\log(r)}\Bigg).
 	\end{align*}
 	Let us fix $0<r_0\leq \mathbf{R}$ small enough such that 
 	\begin{equation}\label{condition1}
 		\Bigg|\frac{\log(\sum_{j=k}^d \theta_j r_0^{j-k})}{\log(r_0)}\Bigg|<1/4.
 	\end{equation}
 	From \citet[Th. 1]{cuepat18} we know ${\mathcal P}_n(r_0)\to V(r_0)$ as $n\to \infty$  a.s. Then, with probability one for $n$ large enough,
 	\begin{equation} \label{condition2}
 		\Bigg|\frac{\log(V(r_0)/{\mathcal P}_n(r_0))}{\log(r_0)}\Bigg|<1/4.
 	\end{equation}
 	Combining  \eqref{condition1} and \eqref{condition2}, with probability one for $n$ large enough,
 	$$f\Bigg(d-\frac{\log({\mathcal P}_n(r_0))}{\log(r_0)}\Bigg)=d-k.$$
 	
 \end{proof}

 \begin{remark}
 	Parts (b) and (c) are perhaps the most interesting conclusion of Theorem \ref{th:polvol} as they provide a wide class of sets $S$ (those with polynomial volume function) for which the assumptions imposed in Theorem \ref{th:general} to get consistency can be just replaced by the simpler conditions $\gamma_n/r_n\to 0$ a.s. and $r_n\to 0$.  Part (d) has a rather conceptual, theoretical value. Indeed, our empirical results suggest that the estimator $\widetilde{\dim}$ considered in Theorem \ref{th:polvol}  is not competitive in practice against the other estimators checked in Section \ref{sec:estimators}. Still, for this estimator, consistency can be obtained for a constant value $r_n=r_0$ of the tuning parameter.
 	While such value is not known in advance, as it depends on the unknown dimension value, equation \eqref{condition1} might provide some clues for an iterative procedure to select $r_0$.  
 \end{remark}

 \subsection{Consistency for the correlation dimension estimator}\label{sec:(P_x)}
 
 The following theorem establishes a consistency result for the natural estimator of the correlation dimension.
 
 \begin{theorem} \label{th:cd}
 	Assume that $X$ is such that $\dim_\textnormal{cd}(P)$ is   finite and strictly positive. Then, the estimator $\widehat{\dim}_\textnormal{cd}$ defined in \eqref{dimcd_est} is almost surely consistent, that is
 	
 	$$\dim_\textnormal{cd}(P)=\lim_{n\to\infty}\frac{\log(\widehat{p}_n(r_n))}{\log(r_n)}\quad a.s.,$$
 	provided that we take 
 	\begin{equation}\label{eq:rncd}
 		r_n=\Bigg(\frac{\log n}{n}\Bigg)^{\frac{1}{(1+\beta)\dim_\textnormal{cd}(P)}},
 	\end{equation}
 	with $\beta>0$.   Also,  any other sequence  $r^*_n$ decreasing to zero not faster than the sequence in \eqref{eq:rncd} (that is with $r^*_n\ge Mr_n$  for some $M>0$) provides consistency. 
 	
 \end{theorem}
 \begin{proof}  
 	Let us first take $r_n$ as indicated in \eqref{eq:rncd}. 
 	
 	$$\mathbb{P}\Bigg(\Bigg |\frac{\log(\widehat{p}_n(r_n))}{\log(r_n)}-\frac{\log(p(r_n))}{\log(r_n)}\Bigg|>\epsilon\Bigg)=   \mathbb{P}\Bigg(\Bigg | \log\Bigg 		(\frac{\widehat{p}_n(r_n)}{p(r_n)}\Bigg)\Bigg |>-\epsilon \log(r_n)\Bigg).$$

 	We first note that $\widehat{p}_n(r_n)$ is a U-statistic. Then will use the concentration inequality for U-statistics, given by equation (2) in \citet[Th. A, p. 201]{serfling}. According to the notation of this book, the order of the statistic is $m=2$, the kernel function $h$ is $h(x_1,x_2)={\mathbb I}_{\|x_1-x_2\|\leq r_n}$, the bounds for the value of $h$ are $a=0$, $b=1$, the expectation of the U-statistic is $\theta=p(r_n)$, \bea{its variance is $\sigma^2=p(r_n)(1-p(r_n))$} and the deviation value is $t=p(r_n)/2>0$, for $n$ large enough. Thus, using the above mentioned inequality and noting that the distribution of $U_n-\theta$ is symmetric around 0, we get  
 	\begin{align*}
 		\mathbb{P}\Big( |\widehat{p}_n(r_n) - p(r_n)| >  p(r_n)/2 \Big) \leq & \ 2 \exp \left( \frac{-n \left(p(r_n) \right)^2/8}{2\left( \sigma^2 + \frac{1}{6} \left( 1 - p(r_n) \right)  p(r_n) \right) } \right)\\ 
 		\leq &	\ 2 \exp \left( \frac{-3n p(r_n)}{56} \right)=:a_n.
 	\end{align*}
 	From \eqref{cd2}, for $n$ large enough,  $p(r_n)>r_n^{(1+\beta/2)\dim_\textnormal{cd}(P)}$\bea{. Then, taking $r_n$ as given in \eqref{eq:rncd}, we have that}   $\sum_n a_n<\infty$.
 	
 	Now   we use first-order Taylor expansion of $\log(x)$ around $x=1$. That is
 	$\log(x)=x-1-\frac{1}{2c^2}(x-1)^2$ where $|c-1|<|x-1|$.
 	\begin{multline*} 
 		\mathbb{P}\Bigg(\Bigg| \log\Bigg(\frac{\widehat{p}_n(r_n)}{p(r_n)}\Bigg) \Bigg| > -\epsilon \log(r_n) \Bigg) = \\
 		\mathbb{P}\Bigg(
 		\Bigg| \frac{\widehat{p}_n(r_n)}{p(r_n)} - 1 \Bigg| 
 		\Bigg| 1 - \frac{1}{2c^2} \Bigg(\frac{\widehat{p}_n(r_n)}{p(r_n)} - 1 \Bigg) \Bigg| > -\epsilon \log(r_n) 
 		\Bigg).
 	\end{multline*}
 		
 		Let us denote $A_n=\{\omega:|\hat{p}_n(r_n)/p(r_n)-1|>1/2 \}$. 
 		\begin{multline*}
 			\mathbb{P}\Bigg(\Bigg| \log\Bigg(\frac{\widehat{p}_n(r_n)}{p(r_n)}\Bigg)\Bigg| > -\epsilon \log(r_n) \Bigg) = \\
 			\mathbb{P}\Bigg(\Bigg| \frac{\widehat{p}_n(r_n)}{p(r_n)} - 1 \Bigg|
 			\Bigg| 1 - \frac{1}{2c^2} \Bigg(\frac{\widehat{p}_n(r_n)}{p(r_n)} - 1 \Bigg) \Bigg| > -\epsilon \log(r_n) \cap A_n \Bigg) \\
 			+ \mathbb{P}\Bigg(\Bigg| \frac{\widehat{p}_n(r_n)}{p(r_n)} - 1 \Bigg|
 			\Bigg| 1 - \frac{1}{2c^2} \Bigg(\frac{\widehat{p}_n(r_n)}{p(r_n)} - 1 \Bigg) \Bigg| > -\epsilon \log(r_n) \cap A_n^c \Bigg)  \leq \mathbb{P}(A_n),
 		\end{multline*}
 	where we used that the  second probability is 0 eventually because $-\log(r_n)\to \infty$ and, on $A_n^c$, $c>1/2$ and then  $(1/c^2)<4$. Thus, on $A_n^c$ we have 
 	$$\Bigg|\frac{\widehat{p}_n(r_n)}{p(r_n)}-1\Bigg|\Bigg|1-\frac{1}{2c^2}\Bigg(\frac{\widehat{p}_n(r_n)}{p(r_n)}-1\Bigg)\Bigg|< 1.$$
 	Then, the desired result follows from $\sum_n \mathbb{P}(A_n)<\infty $, as shown at the beginning of the proof\bea{, together with the first Borel–Cantelli lemma.}  
 	Also, the proof  still holds if we take another sequence $r^*_n$ with $r^*_n\leq Mr_n$. Indeed, it is enough to see the convergence $\sum_na_n<\infty$ stands valid if the $r_n$ specified in \eqref{eq:rncd} is replaced  with such $r^*_n$. 
 	
 \end{proof}
 
 \begin{remark}
 	For a different approach to the consistent estimation of correlation dimension see \citet{qiu22}. 
 \end{remark}

 	%
 	%
 %

 \subsection{Pointwise dimension estimation}\label{sec:point}
 
 The notion of  pointwise dimension defined in \eqref{eq:pw} is perhaps less popular than the Minkowski definition of this concept. 
 However, besides their obvious differences both dimension notions, pointwise and Minkowski, are somewhat complementary, since both are suitable for statistical estimation and both provide useful information about Hausdorff dimension, (see part (b) of Theorem \ref{th:pw} below), which is considerably harder to estimate in a direct fashion. In addition, as seen in Theorem \ref{th:pw} below, some standard methods in nonparametrics can be used to derive convergence rates for the estimator. Last but not least, the empirical results provided below suggest that the ``pointwise-based'' estimator proposed at the end of Section \ref{sec:estimators} could be a competitive choice in dimension assessment studies.

 
 

 \begin{theorem}\label{th:pw} Let $S\subset\mathbb{R}^d$ be a compact set and  $\aleph_n=\{X_1,\dots,X_n\}$ be an iid sample from a distribution $P$, whose support is $S$.  
 	
 	(a)  Let $x\in S$ such that $\dim_\textnormal{pw}(x)$, defined in \bea{\eqref{eq:pw}}, exists  and the standardness condition defined in \eqref{estandar} is fulfilled at $x$ \bea{with constants $r_0$, $\delta$ and $d'$}. Then, the estimator $\widehat{\dim}_{\textnormal{pw}}(x)$  is almost surely consistent, that is 
 	\begin{equation}\label{dpwconsist}
 		{\dim}_{\textnormal{pw}}(x)=\lim_{n\to\infty}\frac{\log({\mathbb P}_n(B(x,r_n)))}{\log(r_n)}\quad a.s.,
 	\end{equation}
 	for $r_n=(C\log(n)/n)^{1/d'}$  and  $C>28/(3\delta)$.

 	(b)  Assume now that the ``global'' (for all $x\in S$) standardness condition \eqref{estandar} holds \bea{with constants $r_0$, $\delta$ and $d'$} and $\dim_\textnormal{pw}(x)$ exists for all $x\in S$. Let $\beta>(4d+12)/\delta^2$  and 
 	\begin{equation*} \label{2rn}
 		r_n=\Bigg(\beta\frac{\log n}{n}\Bigg)^{\frac{1}{2d' }}.
 	\end{equation*}  Then, if the convergence in the definition \eqref{eq:pw} of $\dim_{\textnormal{pw}}$ is uniform on $x$, the 
 	consistency in \eqref{dpwconsist} is uniform as well, that is
 	\begin{equation}\label{dpwconsistunif}
 		\sup_{x\in S}\Big|	\widehat{\dim}_{\textnormal{pw}}(x)-\dim_{\textnormal{pw}}(x)\Big|\to 0,\  \mbox{almost\ surely},  \mbox{ as } n\to\infty.
 	\end{equation}
 	As a consequence, we also have, almost surely, as $n\to\infty$,
 	\begin{align}\label{supdw}
 		&	\sup_{x\in S} \widehat{\dim}_{\textnormal{pw}}(x)\to 	\sup_{x\in S} \dim_{\textnormal{pw}}(x)\geq \dim_\textnormal{H}(S),  \ \mbox{and}\\
 		&	\inf_{x\in S} \widehat{\dim}_{\textnormal{pw}}(x)\to 	\inf_{x\in S} \dim_{\textnormal{pw}}(x)\leq \dim_\textnormal{H}(S) \notag
 	\end{align}
 	where $\dim_\textnormal{H}(S)$ denotes the Hausdorff dimension of $S$. 
 \end{theorem}
 \begin{proof}
 	(a)    Let $B=B(x,r_n)$.
 	\begin{align*}
 		\Bigg|\dim_{\textnormal{pw}}(x)-\widehat{\dim}_{\textnormal{pw}}(x)\Bigg|&\leq    \Bigg|\dim_{\textnormal{pw}}(x)-\frac{\log(P(B))}{\log(r_n)}\Bigg|	
 		+ \Bigg|\widehat{\dim}_{\textnormal{pw}}(x)-\frac{\log(P(B))}{\log(r_n)}\Bigg|\ \nonumber \\
 		&=B_n+V_n.
 	\end{align*}
 	Take $\epsilon>0$ and $n$ large enough such that	$B_n<\epsilon/2$.	Observe that

 	$$\Big|\frac{\log({\mathbb P}_n(B))}{\log(\bea{r_n})}-\frac{ \log(P(B))}{\log(\bea{r_n})}\Big|=\Big|\frac{\log({\mathbb P}_n(B)/P(B))}{\log(\bea{r_n})}\Big|.$$

 	Let us know recall the well-known Bernstein inequality: if $Y_1,\ldots,Y_n$ are independent random variables with mean zero such that $|Y_i|\leq M$ for some constant $M$, we have 
 	\begin{equation*}\label{bernstein}
 		{\mathbb P}\Big(\left|\sum_{i=1}^nY_i\right|>\eta\Big)\leq 2\exp\Big(-\frac{\eta^2/2}{\sum_i{\mathbb E}(Y_i^2)\,+\,M\eta/3}\Big).
 	\end{equation*}
 	We will use this for $\eta= n P(B)/2$, $Y_i={\mathbb I}_B(X_i)-P(B)$, and $M=1$. We use also the standardness assumption 
 	to bound $P(B)$. Note also that $\mathbb{E}(Y^2)=P(B)(1-P(B))$.
 	\begin{align*}
 		\mathbb{P}\left( \left| \frac{\mathbb{P}_n(B)}{P(B)} - 1 \right| > \frac{1}{2} \right)
 		&\leq 2 \exp\left( -\frac{n^2 P^2(B)/8}{n P(B)(1 - P(B)) + nP(B)/6} \right) \notag \\
 		&\leq 2 \exp\left( -\frac{3 \delta r_n^{d'} n}{28} \right) =: a_n.
 	\end{align*}
 	
 The series $\sum_n a_n$ is convergent if $r_n=(C\log(n)/n)^{1/d'}$ with $C>28/(3\delta)$. 	Now   we use first-order Taylor expansion of $\log(x)$ around $x=1$. That is
 $\log(x)=x-1-\frac{1}{2c^2}(x-1)^2$ where $|c-1|<|x-1|$.
 \begin{multline*}  
 	\mathbb{P}\Bigg(\Bigg | \log\Bigg (\frac{\mathbb{P}_n(B)}{P(B)}\Bigg)\Bigg |>-\epsilon \log(r_n)\Bigg)\\
 	=\mathbb{P}\Bigg(\Bigg| \frac{\mathbb{P}_n(B)}{P(B)}-1\Bigg| \Bigg|1-\frac{1}{2c^2}\Bigg(\frac{\mathbb{P}_n(B)}{P(B)}-1\Bigg)\Bigg|>-\epsilon \log(r_n) \Bigg).
 \end{multline*}
 The rest of the proof follows as in the proof of Theorem \ref{th:cd}.
 
 (b)	We have
 \begin{multline*}
 	\Bigg|\dim_{\textnormal{pw}}(x)-\widehat{\dim}_{\textnormal{pw}}(x)\Bigg|\leq  \sup_{x\in S} \Bigg|\dim_{\textnormal{pw}}(x)-\frac{\log(P(B(x,r_n)))}{\log(r_n)}\Bigg|\\+
 	\sup_{x\in S} \Bigg|\widehat{\dim}_{\textnormal{pw}}(x)-\frac{\log(P(B(x,r_n)))}{\log(r_n)}\Bigg|=B_n+V_n.
 \end{multline*}
 
 The term $B_n$ is a sort of ``bias term''. Since the convergence in the definition \eqref{eq:pw} of $\dim_{\textnormal{pw}}$ is uniform on $x$  we have $B_n \to 0$. 
 Regarding the ``variability term'' $V_n$, let us prove that $V_n\to 0$ a.s. 	We will make use of the celebrated Vapnik-Cervonenkis inequality. We will in fact use a particular case of this result; see e.g. \citet[Th. 12.8, Cor. 13.2 and Th. 13.3]{pattern} for a proof and more details. 
 
 \
 
 \noindent	\it  \textbf{[VC inequality].}	Let $P$ be a Borel probability measure on $\mathbb{R}^d$. Let $\aleph_n=\{X_1,\dots,X_n\}$ be an iid sample from $P$. Let ${\mathbb P}_n$ be the empirical measure corresponding to $\aleph_n$. Denote by $\mathcal A_n$ the class of all balls in $\mathbb{R}^d$ of radius $r_n$. Then,  for  $n>2^{d+2}$ and $\epsilon \leq 1$,
 $$\mathbb{P}\Big\{\sup_{B \in \mathcal{A}_n} |{\mathbb P}_n(B) - P(B)| > \epsilon\Big\} \leq C (n^{2(d+2)}+1)e^{-2n\epsilon^2},$$
 where $C$ is a constant that does not exceed $4 e^{4 \epsilon+4 \epsilon^2} \leq 4 e^8$.

 \
 
 \rm

 Now, observe that
 
 $$\Big|\frac{\log({\mathbb P}_n(B))}{\log(r_n)}-\frac{ \log(P(B))}{\log(r_n)}\Big|=\Big|\frac{\log({\mathbb P}_n(B)/P(B))}{\log(r_n)}\Big|.$$
 Using the lower boundedness assumption imposed on $P$. 
 \begin{multline*}
 	\mathbb{P}\Bigg(\sup_{B \in \mathcal{A}_n} \Bigg| \frac{\mathbb{P}_n(B)}{P(B)}-1\Bigg|>1/2\Big)\leq 
 	\mathbb{P}\Big(\sup_{B\in \mathcal{A}_n}|{\mathbb P}_n(B)-P(B)|>\inf_{B\in \mathcal{A}_n} P(B)/2\Big) \\
 	\leq  4e^8(n^{2d+4}+1)\exp\Bigg(-n\inf_{{B\in \mathcal{A}_n}}P^2(B)/2\Bigg)
 	\leq  4e^8(n^{2d+4}+1)\exp\Bigg(-n\delta^2r_n^{2d'}/2\Bigg)\\
 	\leq  4e^8\exp\Bigg(-n\delta^2r_n^{2d'}/2+(2d+5)\log(n)\Bigg)=:b_n
 \end{multline*}

 The convergence of $b_n$ follows from the fact that $\beta>(4d+12)/\delta^2$. Again we use a first-order Taylor expansion of $\log(x)$ around $1$.
 
 \begin{multline*}  
 	\mathbb{P}\Bigg(\sup_{B \in \mathcal{A}_n} \Bigg | \log\Bigg (\frac{\mathbb{P}_n(B)}{P(B)}\Bigg)\Bigg |>-\epsilon \log(r_n)\Bigg)=\\
 	\mathbb{P}\Bigg(\sup_{B \in \mathcal{A}_n} \Bigg| \frac{\mathbb{P}_n(B)}{P(B)}-1\Bigg| \Bigg|1-\frac{1}{2c^2}\Bigg(\frac{\mathbb{P}_n(B)}{P(B)}-1\Bigg)\Bigg|>-\epsilon \log(r_n) \Bigg)
 \end{multline*}
 Thus, \eqref{dpwconsistunif} follows as in the proof of Theorem \ref{th:cd}.

 Finally, the proof of \eqref{supdw} follows from the uniform consistency \eqref{dpwconsistunif} and \citet[Prop. 2.1]{you82} which establishes 
 $$ 
 \inf_x\liminf_{r\to 0}\frac{\log(P(B(x,r)))}{\log (r)}\leq	\dim_\textnormal{H}(S)\leq \sup_x\limsup_{r\to 0}\frac{\log(P(B(x,r)))}{\log (r)}. 
 $$
 \end{proof}

 \begin{remark}
 Again, the conclusions (a) and (b) of Theorem \ref{th:pw} stand true if the respective $r_n$ are replaced with other sequences decreasing to zero with slower rates than those indicated in (a) and (b), respectively.

 A simple sufficient condition to ensure the uniform convergence in \eqref{eq:pw} is the existence of positive constants $C_1<C_2$ and $\ell$ such that,  for all $x\in S$, $C_1r^\ell\leq P(B(x,r))\leq  C_2r^\ell$.	
 \end{remark}
 
 	\section{Empirical results}\label{sec:emp}
 
 Here we evaluate the practical performance of some of the discussed estimators on different sets  $S\subset\mathbb{R}^d$ with differing Minkowski dimensions.  For this, we use random samples $\aleph_n=\{X_1,\dots,X_n\}$ generated uniformly on $S$. First, we consider $S$ to be hypercubes of side length one, for different values of $d$ and $\dim(S)$ (details are given below). This allows us to assess how the estimators perform in scenarios where the ambient space dimension ranges from low to moderate, and the Minkowski dimension ranges from similar to considerably lower than that of the ambient space. Then, following the approach of \citet{cam15}, we analyze the performance of the estimators on a synthetic benchmark. This benchmark comprises a set of 13 manifolds, linearly or nonlinearly embedded in higher dimensional spaces. Finally, we evaluate the performance of the considered estimators in the presence of noise in the data.

 Our objective is not to provide a comprehensive comparison of multiple existing methodologies for dimension estimation (for that, we refer the reader to the study by \citet{cam15}). Instead, we focus on the capacity estimator  
 \begin{equation*}\label{dimcap_est}
 	\widehat{\dim}_{\textnormal{cap}}=-\frac{\log(N_{\textnormal{sep}}(\aleph_n,r_n))}{\log(r_n)},
 \end{equation*}
 the correlation dimension estimator

 \begin{equation}\label{dimcd_est}
 	\widehat{\dim}_{\textnormal{cd}}=\frac{\log(\hat p_n(r_n))}{\log(r_n)}
 \end{equation}
 and the global estimator based on the plug-in  estimation  of the pointwise dimension   \begin{equation}\label{estpw}
 	\widehat{\dim}_{\textnormal{pw}}(x)=  \frac{\log({\mathbb P}_n(B(x,r_n)))}{\log(r_n)}.
 \end{equation}
 
 \sloppy
 Additionally, we include in our study the so-called box-counting estimator, commonly discussed in the literature when referring to fractal dimension estimators. 
 This estimator arises from  the fact that the Minkowski dimension 
 \begin{equation*}\label{bcdim}
 	\dim(S)=\lim_{r\to 0 }\frac{\log(N(S, r))}{\log(1/r)}.
 \end{equation*}
 can be equivalently formulated by replacing the covering number $N(S, r)$ with the minimal number of boxes of side length $r$ required to cover the set, denoted as $N_{\textnormal{box}}(S, r)$ (hence the commonly used term ``box-counting dimension''). For a discussion on the equivalence of this definition, see \citet{fal04}. Thus, another natural estimator for the Minkowski dimension is
 \begin{equation}\label{dimbc_est}
 	\widehat{\dim}_{\textnormal{bc}}=\frac{\log(N_{\textnormal{box}}(\aleph_n,r_n))}{\log(1/r_n)},
 \end{equation}
 where $r_n$ is an appropriate sequence of smoothing parameters $r_n\downarrow 0$. From a practical perspective, algorithms  have been developed in the literature to approximate $N_{\textnormal{box}}(\aleph_n,r_n)$. These box-counting algorithms typically involve placing a standard grid of boxes with side length $r_n$ over the embedding space and count the number of boxes containing at least one point from the sample. 

 The estimators $\widehat{\dim}_{\textnormal{bc}}$, $\widehat{\dim}_{\textnormal{cap}}$ and $\widehat{\dim}_{\textnormal{cd}}$ were computed using the R library {\textit{Rdimtools}}, see \citet{you22}, with the functions \verb|est.boxcount|, \verb|est.packing|, and \verb|est.correlation|, respectively.
 Regarding $\widehat{\dim}_{\textnormal{bc}}$, although \eqref{dimbc_est} requires evaluating a ratio (or quotient) of two terms for a carefully selected value of $r_n$, in practice, instead of directly evaluating this ratio, the box-counting dimension is typically estimated by determining the slope of a linear regression of $\log(N_{\textnormal{box}}(\aleph_n, r_n))$ versus $\log(1/r_n)$ over a suitable range of values of $r_n$. Thus, the implemented algorithm automatically selects the values of $r_n$ (50 by default) and handle extreme points internally, enhancing robustness. A similar approach is used for $\widehat{\dim}_{\textnormal{cd}}$, where, instead of computing the ratio in \eqref{dimcd_est} for a given value of $r_n$, the slope of a linear regression of $\log(\hat p_n(r_n))$ versus $\log(r_n)$ over a suitable range of values of $r_n$ is computed. In the case of $\widehat{\dim}_{\textnormal{cap}}$, the library {\textit{Rdimtools}} implements the scale-dependent capacity dimension estimator described in \citet{keg02}, where the values of $r_1$ and $r_2$ in the algorithm are also automatically selected. For further details on the implemented algorithms,  we refer  to the library's documentation.
 
 Regarding the pointwise dimension estimator in  \eqref{estpw}, although it primarily provides a local measure, we have already noted that a global estimator can be derived from it, defined as $\widehat{\dim}_{\textnormal{pw}}(S)=\sup_{x\in S}\widehat{\dim}_{\textnormal{pw}}(x)$. In practice, 
 we compute $\widehat{\dim}_{\textnormal{pw}}(X_i)$ for $i=1,\ldots,n$ and  estimate the Minkowski dimension of $S$ as the 0.9 quantile of these values. Using a  quantile leads to a more robust estimate compared to using the maximum, as it mitigates the influence of outliers and extreme values. As in the other estimators discussed previously, $\widehat{\dim}_{\textnormal{pw}}(X_i)$ is computed for $i=1,\ldots,n$, by fitting a linear regression to $\log({\mathbb P}_n(B(X_i,r_n)))$ versus $\log(r_n)$ over a suitable range of values of $r_n$.

 Table \ref{tab:SQ} summarizes the results for hypercubes with side length one, for various values of $d$ and $\dim(S)$. Specifically, for each set $S\subset\mathbb{R}^d$, we generated $B=50$ samples of size $n=2500$ uniformly on $S$. We report the mean value of each estimator across the $B$ samples, using the following terminology for the columns in the tables: BC for the box-counting estimator, CAP for the capacity estimator, CD for the correlation dimension estimator and PW for the global estimator based on the pointwise dimension estimation. In parentheses, we also show the proportion of times each estimator correctly identifies the corresponding Minkowski dimension, approximating the estimates to the nearest integer, as the methods may yield non-integer values. We observe that both the box-counting estimator and the capacity estimator tend to underestimate the Minkowski dimension, especially in higher dimensions. The correlation dimension estimator provides more accurate estimates of the Minkowski dimension. The pointwise dimension estimator achieves the best results across all dimensions and is the only one to attain 100\% accuracy under all conditions.
 
 \begin{table}[!ht]
 	\centering
 	\begin{tabular}{cccccc}
 		\hline
 		$d$ & $\dim(S)$ & BC & CAP & CD & PW \\ 
 		\hline
 		2 & 2 & 1.74 (1.00) & 1.74 (0.80) & 1.89 (1.00) & 2.15 (1.00) \\ 
 		\hline
 		3 & 3 & 2.44 (0.00) & 2.33 (0.26) & 2.86 (1.00) & 3.16 (1.00) \\ 
 		& 2 & 1.75 (1.00) & 1.74 (0.84) & 1.89 (1.00) & 2.16 (1.00) \\ 
 		\hline
 		4 & 4 & 3.04 (0.00) & 2.80 (0.02) & 3.77 (1.00) & 4.09 (1.00) \\ 
 		& 3 & 2.43 (0.00) & 2.34 (0.32) & 2.86 (1.00) & 3.16 (1.00) \\ 
 		& 2 & 1.75 (1.00) & 1.68 (0.78) & 1.89 (1.00) & 2.16 (1.00) \\ 
 		\hline
 		5 & 5 & 3.57 (0.00) & 3.05 (0.00) & 4.58 (0.80) & 5.01 (1.00) \\ 
 		& 4 & 2.92 (0.00) & 2.96 (0.04) & 3.76 (1.00) & 4.09 (1.00) \\ 
 		& 3 & 2.43 (0.00) & 2.40 (0.32) & 2.86 (1.00) & 3.16 (1.00) \\ 
 		& 2 & 1.75 (1.00) & 1.71 (0.78) & 1.90 (1.00) & 2.15 (1.00) \\ 
 		\hline
 		6 & 6 & 4.25 (0.00) & 3.41 (0.00) & 5.39 (0.10) & 5.86 (1.00) \\ 
 		& 5 & 3.44 (0.00) & 3.08 (0.00) & 4.60 (0.90) & 4.99 (1.00) \\ 
 		& 4 & 2.92 (0.00) & 2.83 (0.02) & 3.75 (1.00) & 4.10 (1.00) \\ 
 		& 3 & 2.43 (0.00) & 2.31 (0.20) & 2.86 (1.00) & 3.16 (1.00) \\ 
 		& 2 & 1.75 (1.00) & 1.76 (0.76) & 1.90 (1.00) & 2.15 (1.00) \\ 
 		\hline  
 		7 & 7 & 5.33 (0.00) & 3.91 (0.00) & 6.17 (0.00) & 6.66 (1.00) \\ 
 		& 6 & 3.86 (0.00) & 3.33 (0.00) & 5.41 (0.20) & 5.85 (1.00) \\ 
 		& 5 & 3.44 (0.00) & 3.17 (0.00) & 4.59 (0.78) & 5.01 (1.00) \\ 
 		& 4 & 2.92 (0.00) & 2.90 (0.06) & 3.75 (1.00) & 4.10 (1.00) \\ 
 		& 3 & 2.43 (0.00) & 2.29 (0.30) & 2.85 (1.00) & 3.15 (1.00) \\ 
 		& 2 & 1.75 (1.00) & 1.61 (0.64) & 1.89 (1.00) & 2.16 (1.00) \\ 
 		\hline
 	\end{tabular}
 	\caption{Mean values of the box-counting  estimator (BC), capacity  estimator (CAP), correlation dimension estimator (CD) and the global estimator based on the pointwise dimension (PW) over $B=50$ samples of size $n=2500$. Samples are generated uniformly on hypercubes $S\subset\mathbb{R}^d$ for $d=2,\ldots,7$, with Minkowski dimension ${\dim(S)=2,\ldots,d}$. In parentheses, the proportion of times that each estimator correctly estimates $\dim(S)$.}
 	\label{tab:SQ}
 \end{table}

 Table \ref{tab:HA} presents the results on the synthetic benchmark. In order to maintain the same conditions as in \citet{cam15}, we generated $B=20$ samples of size $n=2500$ uniformly on each  manifold $\mathcal{M}_i$, $i =1,\ldots, 13$. The manifolds considered cover a diverse range of structures. They include a $(d-1)$-dimensional sphere linearly embedded ($\mathcal{M}_1$),  affine spaces ($\mathcal{M}_2$ and $\mathcal{M}_9$), a 2-dimensional
 helix ($\mathcal{M}_5$), a Swiss-Roll ($\mathcal{M}_7$) and various nonlinear manifolds ($\mathcal{M}_4$, $\mathcal{M}_6$ and $\mathcal{M}_8$), among others. For a detailed description of the manifolds, we refer to Table 1 in \citet{cam15}, where the same notation is used to facilitate direct comparison. For generating the samples in the synthetic benchmark, we used  the tool\footnote{Available at \url{https://www.ml.uni-saarland.de/code/IntDim/IntDim.html}} developed alongside \citet{hei05}. The box-counting  estimator, which slightly outperforms in some instances the capacity  estimator for low-dimensional manifolds, exhibits a more erratic behavior when applied to higher-dimensional manifolds. The correlation dimension estimator and the pointwise dimension estimator show the best performance overall. Even so, while both tend to underestimate the true dimension in high-dimensional cases, the pointwise dimension estimator seems to perform slightly better in some instances (e.g., $\mathcal{M}_1$, $\mathcal{M}_9$, $\mathcal{M}_{10}$, or $\mathcal{M}_{12}$).

 \begin{table}[!ht]
 	\centering
 	\begin{tabular}{cccrrrr}
 		\hline
 		Manifold&$d$ & $\dim(\mathcal{M}_i)$ & \multicolumn{1}{c}{BC} & \multicolumn{1}{c}{CAP} & \multicolumn{1}{c}{CD} & \multicolumn{1}{c}{PW} \\ 
 		\hline
 		$\mathcal{M}_1$ & 11 & 10 & 10.48 (0.55) & 6.65 (0.00) & 9.06 (0.00) & 9.99 (1.00) \\ 
 		$\mathcal{M}_2$ & 5 & 3 & 2.3 (0.00) & 2.2 (0.15) & 2.89 (1.00) & 3.15 (1.00) \\ 
 		$\mathcal{M}_3$ & 6 & 4 & 2.68 (0.00) & 2.7 (0.05) & 3.59 (0.90) & 4.36 (0.95) \\ 
 		$\mathcal{M}_4$ & 8 & 4 & 4.11 (1.00) & 4.07 (0.90) & 3.79 (1.00) & 4.36 (1.00) \\  
 		$\mathcal{M}_5$ & 3 & 2 & 1.85 (1.00) & 1.68 (0.85) & 1.99 (1.00) & 2.15 (1.00) \\  
 		$\mathcal{M}_6$ & 36 & 6 & 12.15 (0.00) & 5.56 (0.55) & 5.79 (0.95) & 7.27 (0.00) \\ 
 		$\mathcal{M}_7$ & 3 & 2 &  2.1 (1.00) & 2.41 (0.60) & 1.98 (1.00) & 2.15 (1.00) \\  
 		$\mathcal{M}_8$ & 72 & 12 & 25.03 (0.00) & 8.54 (0.00) & 11.69 (0.90) & 16.02 (0.00) \\ 
 		$\mathcal{M}_9$ & 20 & 20 & 26.88 (0.00) & 8.74 (0.00) & 14.43 (0.00) & 16.51 (0.00) \\ 
 		$\mathcal{M}_{10}$ & 11 & 10 & 2.45 (0.00) & 5.43 (0.00) & 8.34 (0.00) & 9.12 (0.05) \\ 
 		$\mathcal{M}_{11}$ & 3 & 2 &  1.95 (1.00) & 2.65 (0.50) & 2.02 (1.00) & 2.16 (1.00) \\ 
 		$\mathcal{M}_{12}$ & 20 & 20 & 10.71 (0.00) & 7.16 (0.00) & 14.04 (0.00) & 18.95 (0.00) \\
 		$\mathcal{M}_{13}$ & 13 & 1 &0.96 (1.00) & 1.26 (0.95) & 1.25 (1.00) & 1.25 (1.00) \\ 
 		\hline
 	\end{tabular}
 	\caption{Mean values of the box-counting dimension estimator (BC), capacity dimension estimator (CAP), correlation dimension estimator (CD) and the global estimator based on the pointwise dimension (PW) over $B=20$ samples of size $n=2500$. Samples are generated uniformly on manifolds $\mathcal{M}_{i}$, $i=1,\ldots, 13$ in $\mathbb{R}^d$ with Minkowski dimension $\dim(\mathcal{M}_{i})$. In parentheses, the proportion of times that each estimator correctly estimates $\dim(\mathcal{M}_{i})$.}
 	\label{tab:HA}
 \end{table}

 To conclude, we analyze the behavior of the estimators in a noisy context. More specifically, we again consider $S$ to be hypercubes of side length one, for different values of $d$ and $\dim(S)$. Unlike the previous scenarios, this time, we add \(d\)-dimensional Gaussian noise \(N_d(0, \sigma^2 I_d)\) to the uniform samples $\aleph_n=\{X_1,\dots,X_n\}$, for various values of \(\sigma\). The results are shown in Figure \ref{fig:noise}. Note that in the noiseless model with \(\sigma = 0\) (solid line), the points coincide with the values in Table \ref{tab:SQ}. Furthermore, in this context, a perfect estimate would result in the points lying on the diagonal in each plot. The pointwise dimension estimator comes closest to this ideal behavior, followed by the correlation dimension estimator. For the other two estimators, as previously mentioned, it can be observed that they both tend to underestimate the true dimension. On the other hand, as expected, when noise is introduced, the estimated values increase, indicating that the data now  reside  in the ambient dimension. Nonetheless, for low noise levels, the pointwise dimension estimator still provides reasonable estimates of \(\dim(S)\), demonstrating its robustness in the presence of noise.

 \begin{figure}[!h]  
 	\centering  
 	\includegraphics[width=0.79\textwidth]{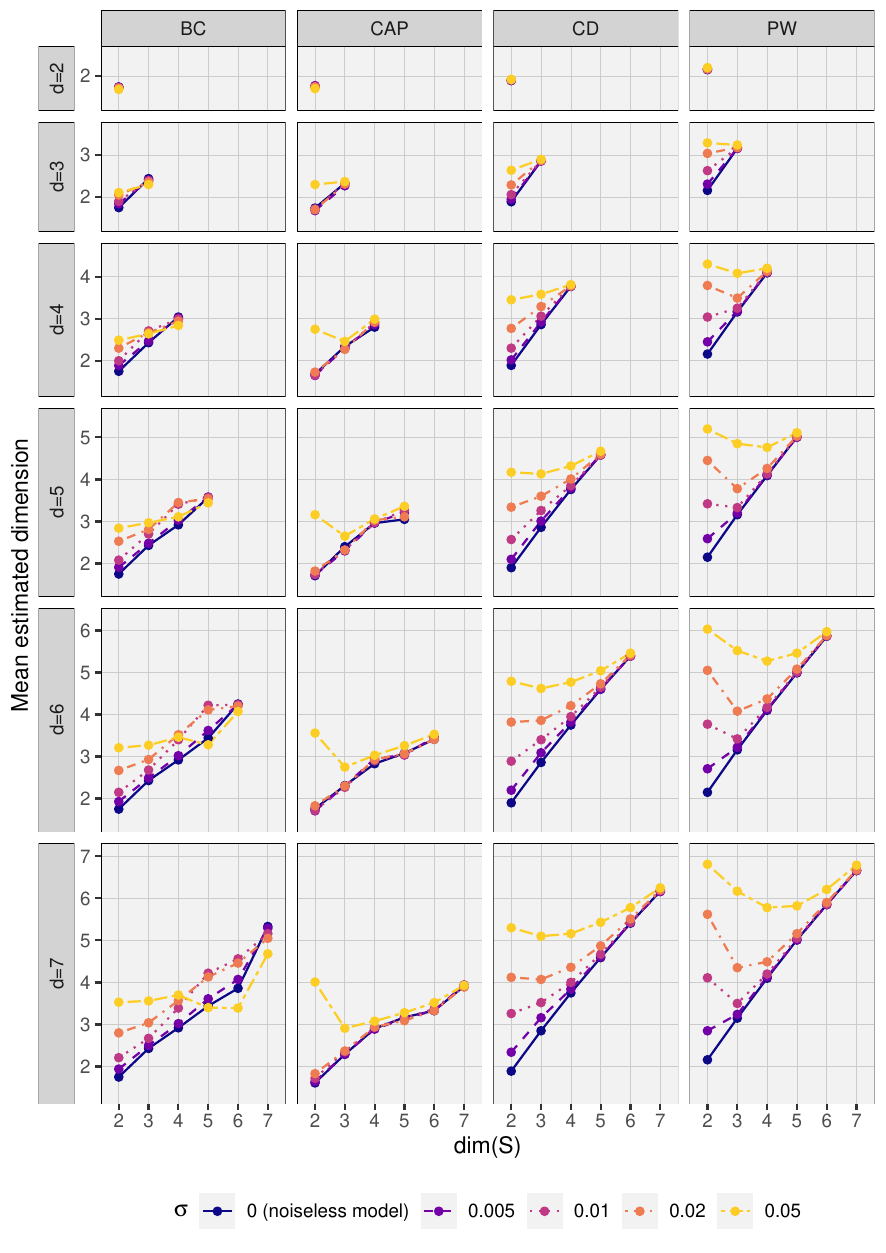}  
 	\caption{Mean values of the box-counting  estimator (BC), capacity  estimator (CAP),  correlation dimension estimator (CD)  and the global estimator based on the pointwise dimension (PW) over $B=50$ samples of size $n=2500$. For the noiseless model (solid lines), samples are generated uniformly on hypercubes $S\subset\mathbb{R}^d$ for $d=2,\ldots,7$, with Minkowski dimension $\dim(S)=2,\ldots,d$. For the noisy settings, $d$-dimensional Gaussian noise $N_d(0,\sigma^2 I_d)$ is added to the the samples with $\sigma=0.005$ (dashed lines), $\sigma=0.01$ (dotted lines), $\sigma=0.02$ (dot-dashed lines), and $\sigma=0.05$ (two-dashed lines).}  
 	\label{fig:noise}  
 \end{figure}

 \section{Conclusions}\label{sec:conclusions}
 
 We have proved consistency results and convergence rates, for three estimators of the dimension of a set $S$ proposed in \citet{keg02},   \citet{gras83} and \citet{you82} and when the sample data is made of random observations on $S$. Here ``consistency'' must be understood in the statistical sense, meaning stochastic convergence to the true value as the sample size tends to infinity. Our methodology is based on techniques similar to those typically employed in nonparametrics, including the use of a sequence smoothing parameters $r_n$ tending to zero ``slowly enough''. Our proofs crucially rely on some auxiliary ``volume based'' estimators, defined in terms of the volume of the $r_n$-dilated sample. Under the, not very restrictive, assumption that the underlying set $S$ has a polynomial volume function on some compact interval $[0,\mathbf R]$, we are able to derive a more informative results with a explicit choice of the sequence of smoothing parameters.
 
 From a more practical point of view, our experimental results suggest that the pointwise dimension estimator proposed in \citet{you82} might be a competitive choice in the dimension estimation problem. 
 
 It should be noted that the aim of the theoretical results in Section \ref{sec:consistencia} is different from (though supplemental to) that of the empirical study in Section \ref{sec:emp}. In Section \ref{sec:consistencia} we are concerned with asymptotic results. So, the goal is to make sure than we are in fact using consistent estimators and to establish the order of convergence in the smoothing parameters $r_n$ that would ensure consistency. No attempt is made to establish optimality for the choice of $r_n$. This would be a much more complicated issue, worth of further study. In the empirical Section \ref{sec:emp} we analyse, via simulations, the practical performance of the considered estimators for a given sample size. Thus,  the use the data-driven smoothing methods implemented in the available software, and considered in the previous literature, appeared as the most sensible choice. 
 
 \section*{Acknowledgments}
 This work of A. Cholaquidis was supported by Grants: CSIC I+D 2022, Comisi\'on Sectorial de Investigaci\'on Cient\'{\i}fica - Universidad de la Rep\'ublica. The research of A. Cuevas has been supported by Grants  
 PID2023-14808-1NB100 from the Spanish Ministry of Science and Innovation and Grant
 CEX2023-001347-S funded by MI-
 CIU/AEI/10.13039/501100011033.  B. Pateiro-L\'opez acknowledges the support of Grants PID2020-116587GB-I00 and ED431C 2021/24, the latter from Conseller\'{i}a de Cultura, Educaci\'{o}n e Universidade.

 \section{An auxiliary result regarding the definition of $\dim_\textnormal{cd}$}
 The definition of $\dim_\textnormal{cd}$ depends (unlike other popular notions of dimension)  on a given probability measure $P$. The following lemma, mentioned in the main document immediately after the definition of $\dim_\textnormal{cd}$, states that the final output is in fact the same for all possible choices of $P$ fulfilling some regularity conditions.

 \begin{lemma}\label{lemma:cd}  Assume  that there exists  a Borel measure $\nu$ and positive constants $C_1<C_2$, $\ell$ and $r_0$ such that,  for a given Borel set $S$, we have for all $x\in S$, $C_1r^\ell\leq \nu(B(x,r))\leq  C_2r^\ell$, for all $r<r_0$. Assume also that $P$ fulfils the ``standardness'' condition $P(B(x,r))\geq \gamma \nu(B(x,r))$ for some $\gamma>0$  and $r$ small enough (say $r<r_0$). Suppose finally that $P$ has a bounded density $g$ with respect to $\nu$.  Then,  	$\dim_{\textnormal{cd}}(P)=\ell$.
 \end{lemma}
 \begin{proof}  First, observe that,
 	
 	\begin{align*}
 		p(r)&=\int_S\mathbb{P}(d(X_1,x)<r|X_2=x)dP(x)\\
 		&=\int_{M}  \nu(B(x,r)) \frac{1}{\nu(B(x,r))} P(B(x,r)) dP(x).
 	\end{align*}
 	Using the standardness assumption we have for $r<r_0$,
 	\begin{align*}
 		\gamma c_1 r^\ell  \leq p(r)&\leq  C_2r^\ell \int_{S}  \frac{1}{\nu(B(x,r))} P(B(x,r)) dP(x)\\ 
 		& =C_2r^\ell \int_{S}  \Bigg[\frac{1}{\nu(B(x,r))} \int_{B(x,r)}g(y)d\nu(y)\Bigg]dP(x).
 	\end{align*}
 	Since  $c_1r^\ell\leq \nu(B(x,r))\leq  C_2r^\ell$, $\nu$ is a doubling measure. Then, by Theorem 1.8 in \citet{heinonen} 
 	$$\Bigg[\frac{1}{\nu(B(x,r))} \int_{B(x,r)}g(y)d\nu(y)\Bigg]\to g(x),$$
 	for almost all $x\in S$ with respect to $\nu$.    Then, by the Dominated Convergence Theorem 
 	$$\int_{S}  \Bigg[\frac{1}{\nu(B(x,r))} \int_{B(x,r)}g(y)d\nu(y)\Bigg]dP(x)\to \int_S  g(x) d\nu(x)=L>0.$$
 	Finally, for $r$ small enough, $	\gamma C_1 r^\ell  \leq p(r)\leq 2C_2Lr^\ell$, from where it follows that $\log(p(r))/\log(r)\to \ell$.
 	
 \end{proof}

\end{document}